\documentclass{amsart}
\usepackage{tikz}
\usepackage{xcolor}
\usepackage{amssymb,latexsym,amsmath,extarrows}
\usepackage{graphicx,mathrsfs}

\newtheorem{theorem}{Theorem}[section]
\newtheorem{lemma}[theorem]{Lemma}
\newtheorem{proposition}[theorem]{Proposition}

\newcommand{\al}{\alpha}
\newcommand{\be}{\beta}

\newcommand{\de}{\delta}

\newcommand{\e}{\varepsilon}

\newcommand{\la}{\lambda}

\newcommand{\vp}{\varphi}

\newcommand{\cq}{\mathcal Q}

\newcommand{\cn}{\mathcal N}

\newcommand{\f}{\frac}

\newcommand{\wh}{\widehat}

\newcommand{\ZR}{\mathbb{R}}
\newcommand{\ZZ}{\mathbb{Z}}

\newcommand{\ZN}{\mathbb{N}}
\newcommand{\ti}{\tilde}

\newcommand{\bB}{{\bf B}}

\newcommand{\bS}{{\bf S}}

\newcommand{\cQ}{{\mathcal Q}}

\begin{document}

\title{New estimates of the Maximal Bochner-Riesz operator in the plane }
\author{Xiaochun Li}

\address{
Xiaochun Li\\
Department of Mathematics\\
University of Illinois at Urbana-Champaign\\
Urbana, IL, 61801, USA}

\email{xcli@math.uiuc.edu}
\author{Shukun Wu}

\address{
Shukun Wu\\
Department of Mathematics\\
University of Illinois at Urbana-Champaign\\
Urbana, IL, 61801, USA}

\email{shukunw2@illinois.edu}
\date{}

\begin{abstract}
We prove new $L^p$-estimates with $1<p<2$ for the maximal Bochner-Riesz operator in the plane.  
\end{abstract}
\maketitle

\section{Introduction}
\setcounter{equation}0

Bochner-Riesz multiplier operator is defined in $\mathbb{R}^n$ as 
$$T_t^\lambda f(x)=(2\pi)^{-n}\int_{\mathbb{R}^n} \Bigg(1-\frac{|\xi|^2}{t^2}\Bigg)^\lambda_+\wh{f}(\xi)e^{ix\cdot\xi}d\xi\,,$$
and the associated maximal operator is given by 
\begin{equation}
\label{MBR}
   T^{ \lambda}_\ast f(x)=\sup_{t>0}|T_t^\la f(x)|.
\end{equation}


The study of the maximal Bochner-Riesz operator $T^\la_\ast$ is closely related to  pointwise convergence of the Bochner-Riesz mean $T^\la_tf$ as $t\to\infty$, for any $f\in L^p(\mathbb R^n)$ with $1<p<\infty$.  When $p>2$, the pointwise convergence phenomena are well understood. For instance, 
for $f\in L^p(\ZR^n)$ with $n>2$ and $p\geq 2$, the pointwise convergence for $T^\la_t f$ in the optimal range of $\la$ was shown by Carbery, Rubio de Francia and Vega in \cite{CRV}, via  power weighted $L^2$-estimates.  For $p\geq 2$,  the $L^p$-estimate of $T^\la_*$ in planar case is completely settled by Carbery \cite{C}, who proved the endpoint case $p=4$ for $T^\la_\ast$. The  the best known $L^p$-results with $p\geq 2$  in the higher dimensional cases is due to Lee \cite{L1}. 
For $p>2$, various endpoint bounds for maximal functions associated with radial Fourier multipliers can be found in \cite{L-S}.\\

However, it remains widely open what the smallest $\lambda$ shall be in order to make $T^\lambda_t f$ converge to $f$ almost everywhere
as $t\rightarrow \infty$, for any $f\in L^p(\mathbb R^n)$ with $1<p<2$.   For $p$ between $1$ and $2$, the pointwise convergence problem, due to Stein's maximal principle \cite{St}, is equivalent to build up weak-$( p, p)$ estimate of the maximal Bochner-Riesz operator $T^{\lambda}_*$.
It was conjectured by Tao \cite{T1} that for any $\lambda > \frac{2n-1}{2p}-\frac{n}{2}$, the maximal Bochner-Riesz operator $T^{\lambda}_*$
satisfies
 $L^p$ estimate, i.e., 
\begin{equation}\label{1}
\big\| T^{\lambda}_\ast f\big\|_{p} \lesssim \|f\|_p\,. 
\end{equation}
Tao provided an example proving that \eqref{1} does not hold for $\la<\frac{2n-1}{2p}-\frac{n}{2}$. In addition,  he proved that in the planar case the $L^p$ estimate (\ref{1}) holds provided that
\begin{equation}\label{T-b}
\lambda > \max\big\{\frac{3}{4p}-\frac{3}{8}, \frac{7}{6p}-\frac{2}{3}\big\}\,.
\end{equation}


Tao's bounds (\ref{T-b}) can be strengthened as indicated in Figure 1, where 
the shaded triangular region corresponds to the new estimates we will obtain in this article. 

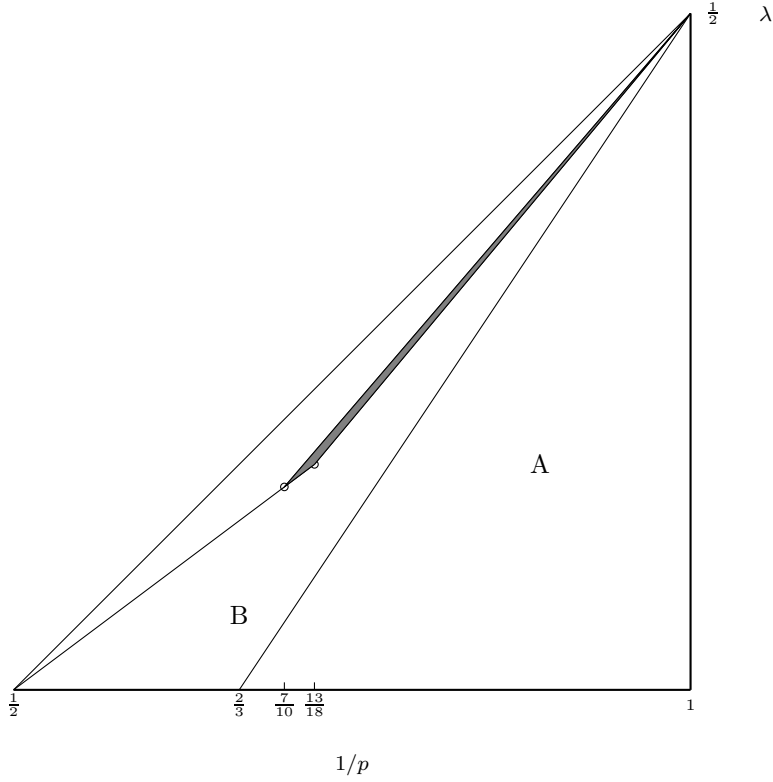
\begin{figure}
\begin{tikzpicture}
\draw[thick] (0,0) -- (9,0) node [anchor=north west] {};
\draw[thick] (9,0) -- (9,9) node [anchor=north east] {};
\node at (4.5,-1) {\footnotesize $1/p$};
\node at (0,-0.2) {\tiny $\frac{1}{2}$};
\node at (3,-0.2) {\tiny $\frac{2}{3}$};
\node at (9,-0.2) {\tiny $1$};
\node at (9.3,9) {\tiny $\frac{1}{2}$};
\node at (10,9) {\footnotesize $\la$};
\draw (3,0) -- (9,9);
\draw (0,0) -- (9,9);
\node at (7,3) {A};
\node at (3,1) {B};
\node at (3.6,-0.2) {\tiny $\frac{7}{10}$};
\draw (3.6,0) -- (3.6,0.1);
\node at (4,-0.2) {\tiny $\frac{13}{18}$};
\draw (4,0) -- (4,0.1);
\draw (3.6,2.7) circle (.5mm);
\draw (4,3) circle (.5mm);
\draw (3.6,2.7) -- (9,9);
\draw (4,3) -- (9,9);
\draw (4,3) -- (0,0);
\draw [fill=gray] (9,9) to (4,3) to (3.6,2.7) to  (9,9) ;
\end{tikzpicture}
\caption{$L^p$ behavior for the maximal Bochner-Riesz operator \eqref{MBR}} 
\label{fig1}
\end{figure}

\begin{theorem}\label{main}
Let $n=2$ and $1<p<2$. Then 
(\ref{1}) holds if $\lambda$ obeys
\begin{equation}
\label{la-p-range}
 \lambda > \la_p=\max\big\{\frac{3}{4p}-\f38, \frac{6}{5p}-\frac{7}{10} \big\}\,.
\end{equation}
\end{theorem}

 Tao's counterexample shows that \eqref{1} fail in the region $A$ in Figure \ref{fig1}.  As shown in the figure, we only improve Tao's result \eqref{T-b} in the range $1<p<10/7$. For $10/7<p<2$, our result coincides with Tao's. The region $B$ remains open. We prove Theorem \ref{main} by combining Tao's method in \cite{T1} with  Bourgain-Demeter's $l^2$-decoupling theorem \cite{BD}.  \\


Throughout the paper,  we use $a\sim b$ to stand for that $ca\leq b\leq Ca$ for some unimportant constants $c$ and $C$. We also use $a\lesssim b$ to represent $a\leq Cb$ for the unimportant constant $C$.
We let $M,N$ be  absolute (big) constants and let $C_\e$ be a constant depending on $\e$. Notice that $C_\e$ may vary from line to line. \\
 
 \noindent
 {\bf Acknowlegement}. The first author is supported by a Simons fellowship in mathematics. 
 The authors would like to express their gratitude to the referee for reading this article carefully and 
 providing numerous valuable suggestions, which help us greatly improve presentation of the paper. \\

\section{A standard reduction}
\setcounter{equation}0

Let $B^n(x,R)$ denote the ball in $\mathbb R^n$,  of radius $R$,  centered at $x$. We let $\{\psi_k\}_{k=0}^\infty$ form a partition of unity of $\ZR^n$ such that $\psi_0$ is a smooth function supported in $B^n(0, 2^M)$; for $k\geq 1$, $\psi_k$ is a smooth function supported on the annulus $A_k =\{x\in \mathbb R^n:  2^{k-1+M}\leq |x|< 2^{k+1+M}\}$.  It is easy to see that $\wh\psi_k(\xi)$ is concentrated on $B(0,2^{1-k-M})$ because $|\partial^{\al}\wh\psi_k|\leq C_{N_\e}2^{n(k+1+M)}(1+|2^{k+1+M}\xi|)^{-N_\e}$ for some large $N_\e$ depends on $\e$, and $\al$ is a multi-index with $|\al|\leq N_\e/2$.
 We define
\begin{equation}
\label{partition-of-unity}
    \psi_{k, t} (x) = \psi_{k}(tx)\,,
\end{equation}
and 
\begin{equation}
K_t^\lambda(x)=\int\left(1- \frac{|\xi|^2}{t^2} \right)^\lambda_+ e^{ix\cdot\xi}d\xi \,.
\end{equation}

Using Fourier inversion formula, we are able to represent 
\begin{eqnarray}\nonumber
T_\ast^\lambda f(x)=\sup_{t>0}\big |\sum_{k\geq 0}(K_t^\lambda\psi_{k,t})\ast f(x)\big |\,,
\end{eqnarray}
which is dominated by, using the triangle inequality, 
\begin{equation}
\sup_{t>0}\big |(K_t^\lambda\psi_{0,t})\ast f(x)\big | + \sup_{t>0}\big |\sum_{k\geq 1}(K_t^\lambda\psi_{k,t})\ast f(x)\big |\,.
\end{equation}


The first term  is controlled  by  Hardy-Littlewood maximal function and hence it satisfies $L^p$ estimates.
To treat the second term,     we recall that  the Bochner-Riesz kernel $K_1^\lambda(x)$ has an asymptotic expansion


$$K_1^\lambda(x)\approx|x|^{-(n+1)/2-\lambda}\left(e^{i|x|}\sum_{j=0}^{\infty}a_j|x|^{-j}+e^{-i|x|}\sum_{j=0}^\infty b_j|x|^{-j}\right)$$
for some constants $a_j,b_j$ as $x\to\infty$.
When $|x|$ is sufficiently large, the principal contribution comes from the first term in the asymptotic expansion.
Henceforth, it suffices to consider the kernel
\begin{equation}
\label{kernel-K}
    \tilde{K}_t^\lambda(x)=t^ne^{i|tx|}|tx|^{-(n+1)/2-\lambda}
\end{equation}
with $x\neq 0$,  
and the associated maximal operator
\begin{equation}
    \tilde{S}_\ast^\lambda f(x)=\sup_{t>0}|\sum_{k\geq 1}(\tilde{K}_t^\lambda\psi_{k,t})\ast f(x)|.
\end{equation}


For each $k$, we define the $\tilde{S}_{\ast,k}^\lambda$ as
\begin{equation}
\label{asym-expanded-MBR}
\tilde{S}_{\ast,k}^\lambda f(x)=\sup_{t>0}|(\tilde{K}_t^\lambda\psi_{k,t})\ast f(x)|.
\end{equation}
We shall point out that the cutoff function $\psi_{k,t}$ below may vary from line to line, but they are essentially
the same function,  since the derivative behaviors as well as the Fourier support behaviors for different $\psi_{k,t}$'s remain the same.



It is easy to see that
\begin{equation}
\|\tilde{S}_{\ast,k}^{\la_p} f\|_{p,\infty}\leq C_\e 2^{\varepsilon k}\|f\|_p
\end{equation}
for any $\varepsilon>0$ implies that for $\lambda>\la_p$ \footnote{One can check this by writing $\|f\|_{p,\infty}$ in the dual form
$$\|f\|_{p,\infty}\sim\sup_{0<|E|<\infty}|E|^{-1/p'}\int f\chi_E$$ 

where $1/p'+1/p=1$. See \cite{G} Chapter 1.4 } 
\begin{equation}
\|T^\lambda_* f\|_{p,\infty}\lesssim \|f\|_p\,,
\end{equation}
following from $|\tilde{S}_{\ast,k}^{\la}f|\lesssim 2^{-k(\la-\la_p)}|\tilde{S}_{\ast,k}^{\la_p}f|$ and summing up all the positive $k$'s. \\

Now let us focus on the plane $\mathbb R^2$ by taking $n=2$.  
Let $a(x,y,t)$ be a smooth function supported on the region  $|x-y|\sim\delta^{-1}, |x|,|y|\lesssim\delta^{-1}$ and $t\sim1 $. Here $x, y\in \mathbb R^2$.   We define

\begin{equation}\label{def-S}
    Sf(x,t)=\delta^{3/2}\int e^{it|x-y|}f(y)a(x,y,t)dy\,,
\end{equation}
and 
\begin{equation}\label{S*}
S^*f(x)=\sup_{t\sim 1} \big |Sf(x, t) \big| \,.
\end{equation}



We introduce one more notation in order to state our technical proposition. Write $A\lessapprox B$ if $A\leq C_\be \delta^{-\be} B$ for some $\be$ much smaller than $\e$. 
The following result allows us to reduce the $L^p$-estimate of the maximal Bochner-Riesz operator to a local $L^p$ estimate of the maximal operator
$S^*$.

\begin{proposition}\label{localization-prop}
Suppose that 
\begin{equation} \label{main-result}
    \|S^*f\|_{p,\infty}\lesssim_\e\delta^{-\la_p}\delta^{-\varepsilon}\|f\|_p
\end{equation}
for any $f\in L^p$.  Then 
\begin{equation}
\label{p-p-estimate}
\|\tilde{S}_{\ast,k}^{\la_p} f\|_{p,\infty}\lesssim_\e 2^{\varepsilon k}\|f\|_p\,.
\end{equation}

\end{proposition}

\vspace{0.3cm}

Proposition \ref{localization-prop} was proved by Tao \cite{T2}.  For the convenience of readers, we provide a proof  in the appendix, similar to Tao's 
argument.     Applying Proposition \ref{localization-prop},  we see that 
(\ref{main-result}) implies $L^p$-estimate of $T_*^\lambda$ if $\lambda>\la_p$. 
To conclude our main result, Theorem \ref{main}, via an interpolation argument which will be presented below,
it suffices to prove the following estimate.  \\

\begin{proposition}\label{prop}
Let $S^\ast$ be in \eqref{S*} and let $E$ be any measurable set in $\mathbb R^2$. Then
\begin{equation}\label{r-6}
 \|S^*\chi_E\|_{2,\infty}\lessapprox\delta^{5/18}|E|^{13/18}\,.  
\end{equation}
\end{proposition}

\vspace{0.3cm}

Let us see first how to use \eqref{r-6} to conclude Theorem \ref{main}. Passing to the duality of weak type norm, we have 
\begin{equation}
\label{dual-form}
    \sup_{0<|F|<\infty}|F|^{-1/2}\int_{\ZR^2} (S^*\chi_E)(x)\cdot\chi_F(x)dx\lesssim\|S^*\chi_E\|_{2,\infty}
\end{equation}
Since the support of $S^*\chi_E$ is contained in $B^2(0,C\de^{-1})$, we can assume $F\subset B^2(0,C\de^{-1})$. Using \eqref{r-6} and
  $|F|\leq C\de^{-2}$, we have
\begin{equation}
    \de^{1-2/p'}\sup_{0<|F|<\infty}|F|^{-1/p'}\int_{\ZR^2} (S^*\chi_E)(x)\cdot\chi_F(x)dx\lessapprox \de^{5/18}|E|^{13/18}
\end{equation}
for $p=\frac{18}{13}$. Passing back to the weak type norm by duality, we thus have
\begin{equation}
\label{weak-type-p}
     \|S^*\chi_E\|_{\frac{18}{13},\infty}\lessapprox\de^{-1/6}|E|^{13/18}.
\end{equation}


Notice that we have a standard $L^2$ estimate by a $TT^*$-method (See Lemma \ref{loc} for instance)
\begin{equation}
\label{global-L2}
\|S^*\chi_E\|_2\lesssim |E|^{1/2}.
\end{equation}
Moreover, by inserting absolute values into the integral (\ref{def-S}), we get trivially 
\begin{equation}
    \label{trivial-esti}
    \|S^*\chi_E\|_1\lesssim \de^{-1/2}|E|.
\end{equation}
Now we can invoke the real interpolation between \eqref{weak-type-p}  and \eqref{global-L2}
for the range $18/13<p<2$ to obtain 
\begin{equation}\label{+}
  \|S^\ast f\|_p\lessapprox \de^{-(\frac{3}{4p}-\frac{3}{8})}\|f\|_p \,\,\, {\rm if}\,\,\, \frac{18}{13}<p<2\,.
    \end{equation}
In addition, interpolating between \eqref{weak-type-p} and  \eqref{trivial-esti}  for $1<p<18/13$, we get
\begin{equation}\label{-}
 \|S^\ast f\|_p\lessapprox \de^{-(\frac{6}{5p}-\frac{7}{10})}\|f\|_p, \,\,\, {\rm if}\,\,\, 1<p< \frac{18}{13}\,. 
 \end{equation}
(\ref{main-result}) follows from (\ref{+}) and (\ref{-}).   Therefore we can conclude Theorem \ref{main} by Proposition \ref{localization-prop}. \\


Proposition \ref{prop} is our main result. Sections \ref{Model} and \ref{Pf-prop} are devoted to the proof of Proposition \ref{prop}. \\

\section{Model operator and its $L^p$-estimates}\label{Model}
\setcounter{equation}0


From now on we shall focus on the operators $S$ and $S^\ast$ defined in \eqref{def-S} and \eqref{S*}, respectively. 
The maximal operator $S^*$ can be linearized as 
\begin{equation}\label{Line}
S^*f(x, t)=  Sf(x, t(x)) \,,
\end{equation}
for some measurable function $t(x)$  taking values $\sim 1$. 
Notice that, if we restricted $t$ in a $\delta$-neighbourhood, the operator $|Sf(x,t)|$ behaves essentially the same, because $x, y$ are restricted in a $\delta^{-1}$-ball. This can be rigorously proved by using Taylor's expansion of 
the exponential function $e^{it|x-y|}$. Thus we are led to partition the interval $I=[1,2]$ into $\delta^{-1}$ many subintervals $I_j$, of length $\delta$.  Let $t_j$ denote the center of $I_j$ and define 
\begin{equation}\label{defFj}
F_j = \big\{x\in B^2(0, \delta^{-1}):  t(x)\in I_j \big\}\,.
\end{equation}
Clearly, $F_j$'s are disjoint mutually. 
Then we see that 
\begin{equation}\label{3-3}
 S^*f(x, t) = \sum_{j} Sf(x, t(x)) \chi_{F_j} (x) \sim \sum_{j}Sf(x, t_j) \chi_{F_j}(x) \,.
\end{equation}
Here we can replace $t(x)$ by $t_j$ because $t(x)-t_i = O(\delta)$ when $t(x)\in F_j$. Henceforth, we 
will focus on the right side of (\ref{3-3}), which is the principal contribution of  $S^*f(x, t)$. 
We run into some non-linear operators because $F_j$'s depend on $f$. But the non-linearity will not cause any trouble because there is no any interpolation argument used for the non-linear operators in our proof. \\

Now for each $j$, the Fourier transform of the kernel $\delta^{3/2}e^{it_j|x|}$ is essentially supported in the $\delta$-neighborhood of the circle center at 0 with radius $t_j$, with a bounded $L^{\infty}$ norm.  
We see that the Fourier transform of $ Sf(x, t_j)$ concentrates on a $\delta$-neighborhood of the circle $\{\xi\in\mathbb R^2:  |\xi|=t_j\}$, up to a negligible contribution,  because outside of the $\delta$-neighborhood, the Fourier transform of $S(f, t_j)$ decays rapidly as the frequency variable stays away from the origin. 
Hence we know that the main contribution in frequency space comes from those $\xi$ lying in the 
the annulus $\{\xi\in\mathbb R^2:   t_j-\delta  \leq |\xi| \leq t_j+\delta\}$,
that can be partitioned into $\delta^{-1/2}$ many congruent pieces, each of which 
 is essentially a $\delta^{1/2}\times\delta$  rectangle.  We use $\Omega_j$ to denote the collection of those disjoint $\delta^{1/2}\times\delta$  rectangles, whose union is essentially the annulus mentioned above.  
 For each $\omega\in \Omega_j$,  a rectangle is called {\it dual} to $\omega$ if its longer side is perpendicular to the longer side of $\omega$, and 
 its dimensions are $\delta^{-1/2}\times \delta^{-1}$.   We further 
break the physical space into $\delta^{-1/2}\times \delta^{-1}$ dual rectangles $R$'s. 
$R\times \omega$ is called a tile.  We use $c(R)$ to denote the center of the rectangle $R$. 
For any $\omega$ and given coordinate axes of $\mathbb R^2$ generated by sides of $\omega-c(\omega)$, we can represent it by $\omega=\omega_1\times \omega_2$ where
$\omega_1$ is an interval of length $\delta^{1/2}$ and $\omega_2$  is an interval of length $\delta$.
Let 
$$
\wh{\varphi_\omega}(\xi_1, \xi_2)= \prod_{k=1}^2 \frac{1}{|\omega_k|^{1/2}}\wh \varphi \big( \frac{\xi_k-c(\omega_k)}{|\omega_k|}\big)\,.
$$
Here $(\xi_1, \xi_2)$ is the coordinates of $\xi$ under the given coordinate axes mentioned above, and 
 $\varphi$ is a Schwartz function from $\mathbb R$ to $\mathbb R$ whose Fourier transform is nonnegative, supported in a small interval,
 of length $\kappa$ (a fixed small constant), about the origin in $\mathbb R$, and identically $1$ on another smaller interval around the origin. 
For any tile $s= R_s\times \omega_s$, we define a function $\varphi_s$ adapted to the tile $s$ by 
\begin{equation}
\wh{\varphi_s}(\xi) = e^{2\pi i c(R)\cdot \xi} \wh{\varphi_\omega} (\xi)\,. 
\end{equation}
Let $\bS_j$ be the collection of all possible tiles $R\times \omega$'s with $\omega\in\Omega_j$.   Up to a negligible contribution, we
end up with a wavepacket representation of $ Sf(x, t_j) $, 
\begin{equation}\label{wave}
Sf (x, t_j)\approx  \sum_{s\in \bS_j } \langle f, \varphi_s\rangle \varphi_s(x) \,.
\end{equation}
Such a wavepacket representation can be proved directly as in \cite{LL} (Page 30) or by employing inductively the one-dimensional result in \cite{LT} (Page 698).\\

We shall clarify the meaning of $\approx$ in (\ref{wave}). In fact, in a rigorous way, one gets
$Sf(x, t_j)= c\sum_{s}\langle f, \varphi_s\rangle S\varphi_s(x, t_j)$ by the wavepacket decomposition of 
$f$.  The function $S\varphi_s(x, t_j)$ behaves similar to the function $\varphi_s$, since both of them have Fourier transforms supported in $\omega_s$ and they are essentially equal to the modulation factor $e^{2\pi i c(\omega_s)\cdot x}$ times an $L^2$-normalized bump function on $R_s$.  This is the reason why we are able to replace the $S\varphi_s(x, t_j)$ by $\varphi_s$ and end up with the wavepacket reprensention of (\ref{wave}). 
\\


Combining (\ref{wave}) with (\ref{3-3}), we see that 
\begin{equation}\label{S*-wave}
 |S^*f(x,t)|\approx  \big| \sum_j \sum_{s\in \bS_j}  \langle f, \varphi_s\rangle \varphi_s(x) \chi_{F_j}(x)  \big| :=
 \big| Tf (x)\big|\,.
 \end{equation}
 Here the operator $T$ is our model operator which we shall study carefully in this section. \\

We introduce a positive parameter $\beta_1$ to split the model operator into two parts. 
For any positive number $\beta_1$, we set 
\begin{equation}
\bB_1 = \big\{s\in \cup_{j} \bS_j:   \big| \langle f, \varphi_s \rangle \big| < \beta_1 \big\}
\end{equation}
and 
\begin{equation}
\bB_2 = \big\{s\in \cup_{j} \bS_j:   \big| \langle f, \varphi_s \rangle \big| \geq \beta_1 \big\}.
\end{equation}
We further break $Tf$ into  $ T_1f + T_2f $, where 
\begin{equation}
T_kf (x) = \sum_j \sum_{s\in \bS_j \cap \bB_k}  \langle f, \varphi_s\rangle \varphi_s(x) \chi_{F_j}(x)
\end{equation}
for $k\in\{1, 2\}$.  We will estimate $T_1f $ and $T_2 f$ in $L^6$ space and $L^1$ space, respectively.

\begin{lemma}\label{T-1}
Let $f\in L^2$. Then 
\begin{equation}\label{L6}
\int \big | T_1 f\big|^6 dx\lessapprox \delta^2 \beta^4_1 \|f\|_2^2 \,.
\end{equation}
\end{lemma}

\begin{proof}

The proof relies on the Bourgain-Demeter's $l^2$-decoupling theorem \cite{BD}. Let us state their
result in $\ZR^2$ here. \\


\begin{theorem}[Bourgain-Demeter]\label{ThmBD}
Let $S$ be a parabola $ \{(\xi, \xi^2): |\xi|\leq 1\}$ in $\ZR^2$. Let $\cn_{\de}$ stand for  $\de$-neighbourhood of $S$ and let $P_{\de}$ be the finite overlapping cover of $\cn_{\de}$ using curved region $\theta=\{(\xi,c+\xi^2):\xi\in I_\theta, |c|\leq2\de\}$,
where $I_\theta$ is a $\de^{1/2}$-lattice interval contained in $[-1,1]$. Denote by $f_\theta$ the Fourier restriction of $f$ to $\theta$. If supp$(\wh{f})\subset N_\de(S)$, then for any $\be>0$, we have
\begin{equation}
\label{l2-decoupling-elliptic}
    \|f\|_6\leq C_\be\de^{-\be}\Big(\sum_{\theta\in P_\theta}\|f_\theta\|_6^2\Big)^{1/2}.
\end{equation}
\end{theorem}

\vspace{3mm}

From Theorem \ref{ThmBD} we obtain


\begin{equation}
    \int \big | T_1f\big|^6 dx\leq C_\be\de^{-\be} \sum_j \bigg( \sum_{\omega\in \Omega_j} \big\|  \sum_{\substack{ s\in \bS_j\cap \bB_1 \\ \omega_s=\omega }}\langle f, \varphi_s\rangle \varphi_s \big\|_6^2  \bigg)^3\,,
\end{equation}
which is bounded by, following from H\"older's inequality, 
\begin{equation}\label{6-6}
C_\be\de^{-\be}\delta^{-1}\sum_j  \sum_{\omega\in \Omega_j} \big\|  \sum_{\substack{ s\in \bS_j\cap \bB_1 \\ \omega_s=\omega }}\langle f, \varphi_s\rangle \varphi_s \big\|_6^6\,,
\end{equation} 
since there are $O(\delta^{-1/2})$ many $\omega$'s in each $\Omega_j$. \\

The functions $\varphi_s$'s are supported essentially  in $R_s$'s, which are pairwise disjoint if $\omega_s$'s are fixed as $\omega$. 
Thus, we see that 
$$
(\ref{6-6})\lessapprox \delta^2 \sum_{s\in \cup_j \bS_j\cap \bB_1} \big | \langle f, \varphi_s\rangle\big|^6\lesssim 
\delta^2 \beta_1^4 \sum_{s} \big | \langle f, \varphi_s\rangle\big|^2 \lesssim \delta^2\beta_1^4 \|f\|_2^2\,,$$
yielding the desired estimation (\ref{L6}). 
\end{proof}

We now introduce the second parameter $\beta_2$ to be chosen later.
Let $E$ be a given measurable set in $\mathbb R^2$ and 
let $\mathcal{Q}$ be a collection  of maximal dyadic cubes $Q$'s such that
$$|E\cap Q|\geq\beta_2|Q|^{1/2}\,.$$
We set
\begin{equation}\label{Eti}
 \tilde{E}=E\setminus\bigcup_{Q\in\mathcal{Q}}Q \,.
\end{equation}

\begin{lemma}\label{T-2}
Suppose that $f\in L^2$ is supported in $\ti E$. Then 
\begin{equation}\label{L-1-est}
\big\|T_2 f \big\|_1\lessapprox \beta_1^{-1}\delta^{-1/4}\beta_2 \|f\|_2^2\,.
\end{equation}
\end{lemma}

\begin{proof}

Inserting absolute values and using the definition of $\bB_2$, we get 
\begin{equation}
\big| T_2f(x)\big|
 \lesssim \beta_1^{-1} \sum_j \sum_{s\in \bS_j } \big| \langle f, \varphi_s\rangle\big|^2 \big| \varphi_s(x)\big| \chi_{F_j}(x)\,.
 \end{equation}

Taking $L^1$-norm for both sides, we have 
\begin{equation} \label{1-1}
\big\| T_2f \big\|_1
\lesssim \beta_1^{-1}\delta^{3/4} \sum_j\sum_{s\in\bS_j }|\langle f,\varphi_{s}\rangle|^2  |R_s\cap F_j|\,.
\end{equation}
The right side of (\ref{1-1}) can be represented as 
$$
\beta_1^{-1}\delta^{3/4} \int \int  f(x) \overline{ f(y)} K(x, y) dx dy \,,$$
where the kernel  $K$ is given as 
\begin{equation}
K(x,y)=\chi_{\ti E}(x)  \chi_{\ti E}(y)  \sum_j\sum_{s\in \bS_j}\varphi_{s}(x)\overline{\varphi_{s}(y)}\big|R_{s}\cap F_j\big|\,.
\end{equation}
It is clear that (\ref{L-1-est}) follows from 
\begin{equation}
    \int \int  f(x) \overline{ f(y)} K(x, y) dx dy \leq \|f\|_2\big(\int\big|\int K(x,y)f(y)dy\big|^2dx\big)^{1/2}
\end{equation}
and
\begin{equation}\label{K-est}
\int\big| K(x, y) \big| dy \lessapprox \beta_2\delta^{-1}\, 
\end{equation}
by employing Schur's test for the operator given by $\int K(x, y) f(y) dy$.



We now turn to a proof of (\ref{K-est}).  From the definition of $\varphi_s$,  we see that 
$ |\varphi_s(x)|$ satisfies 
\begin{equation}
\big|\varphi_s(x)\big| \leq |R_s|^{-1/2} \chi_{R_s}  + \sum_{k=1}^\infty 2^{-1000k} |2^kR_{s}|^{-\f12}\chi_{2^k R_s}(x)\,,
\end{equation}
where $2^k R_s$ is a $2^k$-dilation of $R_s$.  
Due to the fast decay factor $2^{-1000k}$, $|\varphi_s(x)|$ can be treated as 
$ |R_s|^{-1/2} \chi_{R_s} $.  
Therefore, up to a negligible term which can be treated similarly, we have



\begin{equation}\label{K-est-1}
\int\big| K(x, y) \big| dy \lessapprox  \delta^{3/2
}\chi_{\ti E}(x) \int \sum_j\sum_{s\in \bS_j}\chi_{R_{s}}(x)\overline{\chi_{R_{s}}(y)}\big|R_{s}\cap F_j\big|  \chi_{\ti E}(y) dy\,.
\end{equation}
We use $\Theta$ to denote the set consisting of all possible directions of $R_s$'s. 
Here the direction of a rectangle means direction of its longer side. 
Clearly there are at most $O(\delta^{-1/2})$ many elements in $\Theta$.  For any $\theta\in\Theta$, we tile $\mathbb R^2$ by 
$\delta^{-1/2}\times \delta^{-1}$-rectangles with direction $\theta$. Let $\mathcal R_\theta$ denote the collection of all possible such rectangles.
Involving Fubini's theorem and (\ref{K-est-1}), we have 
\begin{equation}
 \int\big| K(x, y) \big| dy \lessapprox \delta^{3/2}
 \chi_{\ti E}(x) \int \sum_{\theta\in\Theta} \sum_{R\in \mathcal R_\theta}\chi_{R}(x)\overline{\chi_{R}(y)}\sum_j\big|R\cap F_j\big|  \chi_{\ti E}(y) dy\,.
\end{equation}


Using the disjointness of $F_j$'s, we see that 
\begin{equation}
   \int\big| K(x, y) \big| dy \lessapprox
    \chi_{\ti E}(x) \int \sum_{\theta\in\Theta} \sum_{R\in \mathcal R_\theta}\chi_{R}(x)\overline{\chi_{R}(y)}  \chi_{\ti E}(y) dy\,.
\end{equation}
The localization argument allows us to focus on only those $R$ contained in a $\delta^{-1}$-ball in $\mathbb R^2$.
Hence we can restrict the region of $y$ in a $\delta^{-1}$-neighborhood of $x$.  We partition this neighborhood into annuli $A_k(x)=\{y: |y-x|\sim 2^k\}$'s, $k\leq \log\delta^{-1}$.   Then 
 $$
  \int \sum_{\theta\in\Theta} \sum_{R\in \mathcal R_\theta}\chi_{R}(x)\overline{\chi_{R}(y)}  \chi_{\ti E}(y) dy
 \lesssim \sum_{k\leq \log\delta^{-1}} \int_{A_k(x)\cap \ti{E}}   \sum_{\theta\in\Theta} \sum_{R\in \mathcal R_\theta}\chi_{R}(x)\overline{\chi_{R}(y)}dy\,.
   $$
The crucial observation is that given $x, y$ with $|x-y|\sim 2^k$, there are at most 
$O(\min\{\delta^{-1/2}, \delta^{-1}2^{-k}\})$ many $(\theta, R) \in \Theta\times \mathcal R_\theta$ such that $x\in R$ and $y\in R$. This is because for each $\theta$, there are at most $O(1)$ many $R$'s containing both $x$ and $y$, and the total number of $\theta$'s  that contribute to $\chi_R(x)\chi_R(y)$ is bounded above by $O(\min\{\delta^{-1/2}, \delta^{-1}2^{-k}\})$.
Therefore we can bound 




\begin{equation}\label{K1}
\int\big| K(x, y) \big| dy \lessapprox  \sum_{k\leq \log\delta^{-1}} \min\{\delta^{-1/2}, \delta^{-1}2^{-k}\} \big| A_k(x)\cap \ti E\big|\,.
\end{equation}
Observe that $A_k(x)$ can be covered by $O(1)$ many dyadic $2^k$-cubes, say $Q$'s. Recalling the definition (\ref{Eti}) of $\ti E$, we see that 
\begin{equation}\label{K2}
\big| A_k(x)\cap \ti E\big| \lesssim \sum_{Q}\big| Q\cap \ti E\big| \lesssim \beta_2 2^k\,, 
\end{equation}
because only those $Q$ with $|Q\cap E|\leq \beta_2|Q|^{1/2}$ meet the set $\ti E$ due to the definition 
of $\ti E$. 
Putting (\ref{K1}) and (\ref{K2}) together, we obtain 
$$
\int\big| K(x, y) \big| dy \lessapprox \sum_{k\leq \log\delta^{-1}}\beta_2 \min\big\{\delta^{-1/2}2^k, \delta^{-1} \big\}
\lessapprox \beta_2 \delta^{-1}\,,
$$
yielding the desired (\ref {K-est}), from which (\ref{L-1-est}) follows by Schur's test.  

\end{proof}

\section{Proof of Proposition \ref{prop}}\label{Pf-prop}
\setcounter{equation}0

We provide a proof of Proposition \ref{prop} in this section. First  we state a local $L^2$ result, which was proved by Tao \cite{T1} via a use of $TT^*$ method.  

\begin{lemma}\label{loc}
Let $Q\subset B^2(0,\delta^{-1})$ be a cube, and $\psi_Q(x)$ be the bump function on $Q$. Then 
\begin{equation} \label{eq1}
\big\| S^*(\psi_Qf)\big\|_2\lesssim\delta^{1/2}|Q|^{1/4}\|f\|_2
\end{equation}
holds for any $f\in L^2$. 
\end{lemma}

We aim to control $\|S^*\chi_E\|_{18/13, \infty}$.  Clearly it follows from the triangle inequality that 
\begin{equation}\label{tri-12}
 S^*\chi_E \leq S^{*}\chi_{E\backslash \ti E}  +  S^*\chi_{\ti E}\,,
 \end{equation}
We now prove a weak $L^2$ estimate for $S^{*}\chi_{E\backslash \ti E}$. 

\begin{lemma}\label{lem-w-2}
For any $\alpha>0$, 
\begin{equation}
\label{eq2}
\big | \{x\in \mathbb R^2: S^{*}\chi_{E\backslash \ti E}(x)\geq\alpha\} \big |\lesssim\alpha^{-2}\delta\beta_2^{-1}|E|^2.
\end{equation}
\end{lemma}
\begin{proof}
By the triangle inequality and use the local $L^2$ estimate (\ref{eq1}), we have
\begin{eqnarray*}
&& \|S^*(\chi_{E\setminus\tilde{E}})\|_2\lesssim\|S^*(\sum_{Q\in\mathcal{Q}}\chi_{E\cap Q})\|_2\lesssim\sum_{Q\in\mathcal{Q}}\|S^*(\chi_{E\cap Q})\|_2\\
&&\lesssim\sum_{Q\in\mathcal{Q}}\delta^{1/2}|Q|^{1/4}|E\cap Q|^{1/2}\lesssim\sum_{Q\in\mathcal{Q}}\delta^{1/2}\beta_2^{-1/2}|E\cap Q|=\delta^{1/2}\beta_2^{-1/2}|E|.
\end{eqnarray*}
Thus, from the Chebyshev inequality (\ref{eq2}) follows. 
\end{proof}

On the other hand,  from (\ref{L6}), we get
\begin{equation}\label{T1le}
\big| \big\{ x\in \mathbb R^2: \big| T_1\chi_{\ti E} \big|\geq \alpha \big\}\big| \lesssim \alpha^{-6} \delta^2 \beta_1^4 |E|\,.
\end{equation}
From (\ref{L-1-est}), we see that 
\begin{equation}
    \label{eq1--2}
   \big | \big\{x\in\mathbb R^2:  \big|T_2\chi_{\ti E} \big|\geq \alpha\}\big|\lessapprox \alpha^{-1}\beta_1^{-1}\delta^{-1/4}\beta_2 |E|.
\end{equation}
Using (\ref{tri-12}), (\ref{S*-wave}), (\ref{eq2}), (\ref{T1le}) and (\ref{eq1--2}), we end up with 
\begin{equation}\label{eqq}
|\{x\in \mathbb R^2: |S^*\chi_E(x,t)|\geq \alpha\}| 
\lessapprox \alpha^{-2}\delta\beta_2^{-1}|E|^2+\alpha^{-1}\beta_1^{-1}\delta^{-1/4}\beta_2|E|+\alpha^{-6}\delta^2\beta_1^4|E| \,.
\end{equation}
We can optimize \eqref{eqq} by taking 
$$\beta_1=\alpha \delta^{-13/36}|E|^{1/9}\,,\,\,\,\, \beta_2=\delta^{4/9}|E|^{5/9}\,. $$
Therefore we have 
\begin{equation} \label{eq-15}
    \|S^*\chi_E\|_{2,\infty}\lessapprox\delta^{5/18}|E|^{13/18}\,,
\end{equation}
which completes the proof of Proposition \ref{prop}.\\


\section{Appendix}
\setcounter{equation}0

In this section we present a proof for Proposition \ref{localization-prop}. We only focus on $1<p<2$.
It is sufficient to prove the following weak type estimate with $n=2$.  
\begin{lemma}
Let $\psi_{k,t}$ and $\tilde{K}_t^{\la_p}$ be defined as in \eqref{partition-of-unity} and \eqref{kernel-K}, respectively. Assume \eqref{main-result} is true and $1<p<2$. Then for $k\geq 1$, $\al>0$,
\begin{equation}
\label{level-set-al}
    \big|\{x\in \mathbb R^n: \sup_{t>0}|(\tilde{K}_t^{\la_p}\psi_{k,t})\ast f(x)|\geq \alpha\}\big|\lesssim_\e 2^{\e k}\alpha^{-p}\|f\|_p^p.
\end{equation}
\end{lemma}

  The proof of (\ref{level-set-C}) is based on 
a standard Carlder\'on-Zygmund type argument. 
In this article, we only need to use the result (\ref{level-set-al}) in the planar case $n=2$.
However, the method below works for general $n$ and this is the reason why we employ $n$ instead of $2$.
\\ 

Since the operator $\tilde{S}_{\ast,k}^{\la_p}$ defined in \eqref{asym-expanded-MBR} is a sub-linear operator, we can absorb the magnitude $\alpha$ into $f$, so that it suffices to show
\begin{equation}
\label{level-set-1}
    \big|\{ x\in\mathbb R^n: \sup_{t>0}|(\tilde{K}_t^{\la_p}\psi_{k,t})\ast f(x)|\geq 1\}\big|\lesssim_\e 2^{\e k}\|f\|_p^p.
\end{equation}
For some technical reasons, we will instead prove
\begin{equation}
\label{level-set-C}
    \big|\{ x\in\mathbb R^n: \sup_{t>0}|(\tilde{K}_t^{\la_p}\psi_{k,t})\ast f(x)|\geq C\}\big|\lesssim_\e 2^{\e k}\|f\|_p^p.
\end{equation}
Here $C$ is an absolute big constant depends on $\e$. In fact, $C$ only depends on our choice of the partition of unity \eqref{partition-of-unity}.  The desired estimate \eqref{level-set-1} follows from \eqref{level-set-C} since the operator $\tilde{S}_{\ast,k}^{\la_p}$ is sub-linear.
\vspace{3mm}

Now we begin to prove \eqref{level-set-C}. Let $J_{j}$ be the time interval $[2^j,2^{j+1}],~j\in\mathbb{Z}$. Partitioning $\mathbb{Z}$ into residue classes mod $2k$, we can write 
\begin{equation}
    \ZR^+=\bigcup_{j\in\mathbb{Z}}J_j=\bigcup_{0\leq r\leq 2k-1}\bigcup_{m\in\mathbb{Z}}J_{2km+r}.
\end{equation}
With a translation argument, we can treat $\cup_{m}J_{2km+r}$ similarly for different $r$. Thus, by losing $2k$, which can be absorbed in $2^{\e k}$ in summation over $r$, it suffices to consider the case $r=0$, that is, 
\begin{equation}
    \label{level-set-esti}
    \Big|\Big\{x\in\mathbb R^n: \sup_{t\in\bigcup_m J_{2km}}|(\tilde{K}_t^{\la_p}\psi_{k,t})\ast f(x)|\geq C\Big\}\Big|\lesssim_\e 2^{\e k}\|f\|_p^p\,.
\end{equation}

Next we adapt a Calderon-Zygmund decomposition on $|f|^p$ at level 1 to get\footnote{One can find the decomposition in \cite{St2} Chapter 3.}
$$f=g+\sum_Qb_Q,$$
where

\begin{equation}\label{b-g}
  \|g\|_\infty \lesssim1,
\end{equation}
\begin{equation}
  {\rm supp}(b_Q)\subset Q,
 \end{equation}
 and 
 \begin{equation}\label{s-bq}
  \|b_Q\|_p\lesssim|Q|^{1/p}.
\end{equation}
Here $Q$'s are maximal dyadic cubes and we use $\cQ$ to denote the collection of all possible 
$Q$'s. Moreover, $b_Q$ satisfies the moment conditions
\begin{equation}
\label{moment-condi}
    \int_{\ZR^n}b_Q(x)x^rdx=0, \hspace{1cm} 0\leq |r|\leq N,
\end{equation}
where $r$ is a multi-index in $(\ZN\cup \{0\})^n$.  In the rest of the section, we use 
$c(Q)$ to denote the center of $Q$, and $2^{l(Q)}$ to stand for the side length of the dyadic cube 
$Q$.  \\ 

Because of (\ref{b-g}) and $p\in(1, 2)$, we see that $\|g\|_2^2\lesssim \|g\|_p^2$.
In addition, the operator $\tilde{K}_t^{\la_p}\psi_{k,t}* g$ is bounded in $L^2$ for $\la_p>0$.
Following from the Chebyshev inequality,  it is easy to see that the contribution from the function $g$ satisfies the desired estimate. Henceforth, it remains to prove
\begin{equation}\label{b-Q}
\Bigg|\Bigg\{x\in\ZR^n: \sup_{m\in\ZZ}\sup_{t\in J_{2km}}\Bigg|(\tilde{K}_t^{\la_p}\psi_{k,t})\ast \Bigg(\sum_{Q\in\cq} b_Q\Bigg)(x)\Bigg|\geq C\Bigg\}\Bigg|\lesssim_\e 2^{\e k}
\sum_{Q\in\cq}|Q|.
\end{equation}

\vspace{3mm}

 For fixed $m$, we sort $Q\in\cq$ in terms of the size $l(Q)$ by letting
$$\cQ=E^1_{k,m}\cup E^2_{k,m}\cup E^3_{k,m},$$
where $E^1_{k,m}=\{Q\in\cQ:l(Q)\leq -2km-k\}$, $E^2_{k,m}=\{Q\in\cQ:-2km-k< l(Q)\leq -2km+k\}$, and $E^3_{k,m}=\{Q\in\cQ:-2km+k<l(Q)\}$. 
As a result, we only need to prove \eqref{b-Q} with $\cq$ replaced by $E^j_{k,m},i=1,2,3$.

\vspace{3mm}

We first work on $E^3_{k,m}$. Recall that the kernel $\tilde{K}_t^{\la_p}\psi_{k,t}$ equals to
\begin{equation}
    t^ne^{i|tx|}|tx|^{-(n+1)/2-\lambda}\psi_k(tx).
\end{equation}
Hence for fixed $m$, ${\rm supp}(\tilde{K}_t^{\la_p}\psi_{k,t})\subset B^n(0,C2^{k-2km})$. This implies the support of $(\tilde{K}_t^{\la_p}\psi_{k,t})\ast b_Q$ is contained in $\bigcup_Q CQ$, where $CQ$ is the cube centered at $c(Q)$ and with side length $C2^{l(Q)}$. Combining with the following equation

\begin{eqnarray} 
\label{union-m}
    &&\Bigg\{x\in\ZR^n: \sup_{m\in\mathbb{Z}}\sup_{t\in J_{2km}}\Bigg|(\tilde{K}_t^{\la_p}\psi_{k,t})\ast \Bigg(\sum_{Q\in E^3_{k,m}} b_Q\Bigg)(x)\Bigg|\geq C\Bigg\}\\ \nonumber
    &=&\bigcup_{m\in\ZZ}\Bigg\{x\in\ZR^n:\sup_{t\in J_{2km}}\Bigg|(\tilde{K}_t^{\la_p}\psi_{k,t})\ast \Bigg(\sum_{Q\in  E^3_{k,m}} b_Q\Bigg)(x)\Bigg|\geq C\Bigg\},
\end{eqnarray}
we then have
\begin{equation}
    \Bigg|\Bigg\{x\in\ZR^n:\sup_{m\in\ZZ}\sup_{t\in J_{2km}}\Bigg|(\tilde{K}_t^{\la_p}\psi_{k,t})\ast \Bigg(\sum_{Q\in E^3_{k,m}} b_Q\Bigg)(x)\Bigg|\geq C\Bigg\}\Bigg|\lesssim 
\sum_{Q\in\cQ}|Q|,
\end{equation}
as desired, because the set in the left side is a subset of $\cup_{Q\in\mathcal Q} C Q$.

\vspace{3mm}

Next, we consider $E^1_{k,m}$. We will use the moment conditions \eqref{moment-condi} and  aim to show the set
\begin{equation}
\label{E-1}
    \Bigg\{x\in\ZR^n: \sup_{m\in\mathbb{Z}}\sup_{t\in J_{2km}}\Bigg|(\tilde{K}_t^{\la_p}\psi_{k,t})\ast \Bigg(\sum_{Q\in E^1_{k,m}} b_Q\Bigg)(x)\Bigg|\geq C\Bigg\}
\end{equation}
is an empty set. As a consequence, \eqref{b-Q} follows with $\cq$ replaced by $E^1_{k,m}$. Since the equation \eqref{union-m} is true for $E^3_{k,m}$ replaced by $E^1_{k,m}$, it suffices to show that for fixed $m\in\ZZ$,
\begin{equation}\label{emp}
    \Bigg\{x\in\ZR^n: \sup_{t\in J_{2km}}\Bigg|(\tilde{K}_t^{\la_p}\psi_{k,t})\ast \Bigg(\sum_{Q\in E^1_{k,m}} b_Q\Bigg)(x)\Bigg|\geq C\Bigg\}=\varnothing.
\end{equation}

To see why (\ref{emp}) makes sense, let us take a quick look at what the moment conditions give us in the frequency space. First,  \eqref{moment-condi} yields that $ (D^r\widehat{b}_Q)(0)=0$ for $|r|\leq N$, where $D$ is the differential operator in $\ZR^n$. This tells us that intuitively, in a small neighbourhood of 0, $\widehat{b}_Q(\xi)\approx0$. On the other hand, observe that the Fourier support of $\tilde{K}_t^\lambda\psi_{k,t}$ is concentrated in the ball $B^n(0,2km)$. The radius of the ball $B^n(0,2km)$ is small compare to $l(Q)^{-1}$, for $Q\in E^1_{k,m}$. Therefore, via Fourier inverse formula, the pointwise value $|\tilde{K}_t^\lambda\psi_{k,t}\ast b_Q(x)|$ should be small, from which we see that (\ref{emp}) does make sense. \\

 We now turn to our rigorous argument for making a proof of (\ref{emp}). 
By a translation argument, without loss of generality, we assume $Q$ is centered at the origin. Abbreviating $l(Q)$ as $l$ and using Taylor's formula, we see that for multi-indices $r_1,r_2\in (\ZN\cup \{0\})^n$, $|r_1|=N-1$, $|r_2|=N$, we have
\begin{equation}
\label{derivative-condition}
    |(D^{r_1}\widehat{b}_Q)(\xi)|\leq|\xi|\cdot\|D^{r_2}\widehat{b}_Q\|_\infty\leq|\xi|\cdot\int |x|^{|r_2|}|b_Q|dx\leq|\xi|\cdot2^{lN}\|b_Q\|_1.
\end{equation}
From H\"older's inequality we have
$$\|b_Q\|_1\leq\|\chi_Q\|_{p'}\|b_Q\|_p\lesssim|Q|,$$
which, together with \eqref{derivative-condition}, implies 
\begin{equation}
    |(D^{r_1}\widehat{b}_Q)(\xi)|\lesssim|\xi|\cdot2^{lN}|Q|.
\end{equation}
Noticing $(D^{r}\widehat{b}_Q)(0)=0$ for $|r|\leq N$, we expand $\widehat{b}_Q(\xi)$ into its Taylor series at 0 and use the remainder formula for the $(N-1)$-th term to get
\begin{equation}
\label{hat-b_Q}
    |\widehat{b}_Q(\xi)|\lesssim|\xi|^N2^{lN}|Q|.
\end{equation}

On the other hand, we calculate the Fourier transform of the kernel  $\tilde{K}_t^{\lambda_p}\psi_{k,t}$ to see 
\begin{equation}
\label{kernel-appendix}
    \big|\widehat{\tilde{K}_t^{\la_p}\psi_{k,t}}(\xi)\big|\approx2^{-\la_p k}\Big|\widehat{\psi}\left(\frac{|\xi|-2^{2km}}{2^{2km}2^{-k}}\right)\Big|.
\end{equation}
Here the right side of \eqref{kernel-appendix} is the principal contribution of the Fourier transform of $\ti{K}_t^{\lambda_p}\psi_{k, t}$.
It follows from the pointwise estimate $\big|\widehat{\psi}(|\xi|)\big|\lesssim(1+|\xi|)^{-M}$ that for $|\xi|$ much larger than $2^{2km}$,
\begin{equation}
\label{hat-kernel-K}
    |\widehat{\tilde{K}_t^{\la_p}\psi_{k,t}}(\xi)|\lesssim\Bigg(\frac{|\xi|-2^{2km}}{2^{2km}2^{-k}}\Bigg)^{-M}.
\end{equation}

We choose $M>2N$ and $N$ much larger than $n$. Since $(\tilde{K}_t^{\la_p}\psi_{k,t})\ast b_Q$ is supported in $B(0,C2^{k-2km})$, combining \eqref{hat-b_Q} and \eqref{hat-kernel-K}, we obtain the pointwise estimate
\begin{eqnarray}
\label{pointwise-estimate}
    |(\tilde{K}_t^{\la_p}\psi_{k,t})\ast b_Q(x)| &\lesssim& \|\widehat{\tilde{K}_t^{\la_p}\psi_{k,t}}
    \widehat{b}_Q\|_1\chi_{B^n(c(Q),C2^{-2km+k})}(x)\\ \nonumber
    &\lesssim&2^{lN}2^{2kmN}2^{2kmn} |Q|\chi_{B^n(c(Q),C2^{-2km+k})}(x).
\end{eqnarray}
Since the dyadic cubes in $E^1_{k,m}$ are mutually disjoint, for fixed $x$, there are  at most $O(2^{-2kmn+kn}|Q|^{-1})$ many cubes $Q$'s of length $2^l$ that make contribution to the sum $\sum_{Q\in E^1_{k,m}}(\tilde{K}_t^{\la_p}\psi_{k,t})\ast b_Q(x)$. Hence, we can sum up all the $Q$'s with $l(Q)=2^l$ and invoke \eqref{pointwise-estimate} to have
\begin{equation}
    \sup_{t\in J_{2km}}\Bigg|(\tilde{K}_t^{\la_p}\psi_{k,t})\ast \Bigg(\sum_{Q\in E^1_{k,m},l(Q)=2^l}b_Q\Bigg)(x)\Bigg|\lesssim2^{lN}2^{2kmN}2^{kn}.
\end{equation}

Finally, as $N$ is much larger than $n$, we can sum up all the $l$'s with $l\leq-2km-k$ to get the following pointwise estimate
\begin{equation}
    \label{c-z-type}
    \sup_{t\in J_{2km}}\Bigg|(\tilde{K}_t^{\la_p}\psi_{k,t})\ast \Bigg(\sum_{Q\in E^1_{k,m}}b_Q\Bigg)(x)\Bigg|\lesssim2^{-kN/2}.
\end{equation}
This pointwise estimate 
shows that if we choose $C$ large enough, the set 
\begin{equation}
    \Bigg\{x\in\ZR^n: \sup_{t\in J_{2km}}\Bigg|(\tilde{K}_t^{\la_p}\psi_{k,t})\ast \Bigg(\sum_{Q\in E^1_{k,m}} b_Q\Bigg)(x)\Bigg|\geq C\Bigg\}
\end{equation}
is an empty set, uniformly for all $m\in\ZZ$. Henceforth, we finish the proof of (\ref{emp}), from which \eqref{E-1} follows.

\vspace{3mm}

Finally,  it remains to show the estimate \eqref{b-Q} with $\cq$ replaced by $E^2_{k,m}$.
Since \eqref{union-m} holds for $E^2_{k,m}$ as well, we only need to show the following estimate for a single $m\in\ZZ$
\begin{equation}
\label{weak-t-2km}    
    \Bigg|\Bigg\{x\in\ZR^n:\sup_{t\in J_{2km}}\Bigg|(\tilde{K}_t^{\la_p}\psi_{k,t})\ast\Bigg( \sum_{\substack{Q\in E^2_{k,m}}} b_Q\Bigg)\Bigg|\geq C\Bigg\}\Bigg|\lesssim_\e2^{\e k}\sum_{\substack{Q\in E^2_{k,m} }}|Q|.
\end{equation}
Without loss of generality, we can assume for free that all $Q\in E_{k, m}^2$'s have the same size because 
$l(Q)$ taking at most $O(k)$ many values between $-2km-k$ and $-2km+k$.  
Thus we have the disjointness property of those $Q$'s in $E_{k, m}^2$ and now it is clear that 
\begin{equation}
 \big\| \sum_{Q\in E^{2}_{k,m}} b_Q \big\|_p^p \lesssim \sum_{Q\in E^2_{k,m}} \big|Q\big|
\end{equation}
by (\ref{s-bq}). Hence, 
(\ref{weak-t-2km}) follows from a dilation argument and  the following stronger result
\begin{equation}
\label{strong-p-p}
    \big\|\sup_{t\in[1,2]}|(\tilde{K}_t^{\la_p}\psi_{k,t})\ast f|\big\|_p\lesssim_\e 2^{\e k}\|f\|_p\,.
\end{equation}

(\ref{strong-p-p}) is a consequence of (\ref{main-result}). Indeed, 
let $\vp_1(t)$ be a smooth function supported on $[3/4,9/4]$,  taking value 1 for $t\in[1,2]$. As a result, we can replace the kernel 
$\tilde{K}_t^{\la_p}\psi_{k,t}$ by $\vp_1(t)\tilde{K}_t^{\la_p}\psi_{k,t}$ in \eqref{strong-p-p}.
Notice that in the planar case, when $f$ is supported $B^2(0,\de^{-1})$, the function $(\vp_1(t)\tilde{K}_t^{\la_p}\psi_{k,t})\ast f$ is morally $\de^{\la_p}Sf(x,t)$ in \eqref{def-S}, with $\de=2^{-k}$.
Via a standard localization trick on the support of $f(x)$,  we see that \eqref{main-result} implies \eqref{strong-p-p}. Therefore, we finish the proof of \eqref{level-set-al}.  \qed




\end{document}